\documentclass[12pt,fleqn]{article}

\usepackage{amsmath,amssymb,amsthm}
\usepackage{geometry} 
\usepackage[pdfpagemode=UseNone,pdfstartview=FitH]{hyperref}
\usepackage{color,tikz-cd,float}

\newcommand{\As}[1][]{A_s #1}
\newcommand{\Bs}[1][]{B_s #1}

\newcommand{\bgset}[1]{\big\{#1\big\}}

\newcommand{\F}{{\mathcal F}}
\newcommand{\M}{{\mathcal M}}

\newcommand{\norm}[2][]{\left\|#2\right\|_{#1}}
\renewcommand{\O}{\text{O}}

\newcommand{\PS}[1]{$(\text{PS})_{#1}$}

\newcommand{\R}{\mathbb R}
\newcommand{\restr}[2]{\left.#1\right|_{#2}}
\newcommand{\seq}[1]{\left(#1\right)}
\newcommand{\set}[1]{\left\{#1\right\}}
\newcommand{\wto}{\rightharpoonup}
\newcommand{\Z}{\mathbb Z}

\newenvironment{enumroman}{\begin{enumerate}

}{\end{enumerate}}

\newtheorem{lemma}{Lemma}[section]
\newtheorem{proposition}[lemma]{Proposition}
\newtheorem{theorem}[lemma]{Theorem}

\theoremstyle{definition}
\newtheorem{definition}[lemma]{Definition}

\theoremstyle{remark}
\newtheorem{remark}[lemma]{Remark}

\numberwithin{equation}{section}

\title{On a scaled abstract linking theorem with an application to the Schr\"{o}dinger--Poisson--Slater equation\bf}
\author{\bf Kanishka Perera\\
Department of Mathematics\\
Florida Institute of Technology\\
150 W University Blvd, Melbourne, FL 32901-6975, USA\\
\em kperera@fit.edu\\
[\medskipamount]
\bf Kaye Silva\\
Instituto de Matem\'{a}tica e Estat\'{i}stica\\
Universidade Federal de Goi\'{a}s\\
Rua Samambaia, 74001-970 Goi\^{a}nia, GO, Brazil\\
\em kayesilva@ufg.br}
\date{}

\begin{document}

\maketitle

\begin{abstract} We prove an abstract linking theorem that can be used to show existence of solutions to various types of variational elliptic equations, including  Schr\"{o}dinger--Poisson--Slater type equations.

\end{abstract}

\begin{center}
	\begin{minipage}{12cm}
		\tableofcontents
	\end{minipage}
\end{center}

\newpage

\section{Introduction}
In this work we will develop an abstract critical point theory that can be used to show existence of pairs of solutions to various types of variational elliptic differential equations. Indeed, the theory can be applied, for example, to equations involving the $p$-Laplacian, Kirchhoff, or fractional $p$-Laplacian operators. However, our main application will concern some types of Schr\"{o}dinger--Poisson--Slater equations which we will describe in the next section.   

\subsection{Schr\"{o}dinger--Poisson--Slater Equation} \label{SPSsection}

Consider the following $N$-dimensional Schr\"{o}dinger--Poisson--Slater equation:

\begin{equation} \label{103}
	- \Delta u +\left(I_\alpha \star |u|^p\right) |u|^{p-2}u = \lambda\, |u|^{q-2} u + \mu f(u)+ |u|^{r-2}\, u \quad \text{in } \R^N,
\end{equation}
where $N\ge 3$,  $\alpha\in (1,N)$, $1< p< (N+\alpha)/(N-2)$, $q=q(\alpha,p,N):=2\frac{2p+\alpha}{2+\alpha}$, $r\in (q,2^*)$, $2^*=2N/(N-2)$, $\lambda,\mu>0$ are real parameters and $I_\alpha:\mathbb{R}^N\to \mathbb{R}$ is the Riesz potential of order $\alpha$, defined for $x\in \mathbb{R}^N\setminus\{0\}$ as 
\begin{equation*}
	I_\alpha(x)=\frac{A_\alpha}{|x|^{N-\alpha}},\ \ \ A_\alpha=\frac{\Gamma(\frac{N-\alpha}{2})}{\Gamma(\frac{\alpha}{2})\pi^{N/2}2^\alpha}.
\end{equation*}
The function $f:\mathbb{R}\to \mathbb{R}$ is continuous and satisfies 

\begin{equation}\label{f1}
	|f(t)|\le C|t|^\ell,\ t\in \mathbb{R}\ \ \   \mbox{and}\ \ \ \lim_{t\to 0}\frac{F(t)}{|t|^\ell}=c>0,
\end{equation}
where $F(t)=\int_0^tf(y)dy$, $\ell\in (\tilde{q},q)$, and $\tilde{q}=\tilde{q}(\alpha,p,N):=\frac{2[2p(N-1)+N-\alpha]}{3N+\alpha-4}$. We will also assume
\begin{equation}\label{f2}
	F(t)>0\ \forall t\in \mathbb{R}\setminus\{0\}.
\end{equation}

When $p=\alpha=2$ and $N=3$, equation \eqref{103} has the form 

\begin{equation} \label{10}
	- \Delta u + \left(I_2\star u^2\right) u = h(u) \quad \text{in } \R^3.
\end{equation}
This equation is related to the Thomas--Fermi--Dirac--von\;Weizs\"acker model in Density Functional Theory (DFT), where the local nonlinearity $h$ takes the form $h(u) = u^{5/3} - u^{7/3}$ (see, e.g., Lieb \cite{MR629207,MR641371}, Le Bris and Lions \cite{MR2149087}, Lu and Otto \cite{MR3251907}, Frank et al.\! \cite{MR3762278}, and references therein). More general local nonlinearities also arise in DFT and quantum chemistry models such as Kohn-Sham's, where $h$ is the so-called exchange-correlation potential and is not explicitly known (see, e.g., Anantharaman and Canc\`es \cite{MR2569902} and references therein). 

The $N$-dimensional version of \eqref{10} was proposed by Bao et al. \cite{BaMaSt}. More recently, in Mercuri et al. \cite{MR3568051}, a thorough study of 

\begin{equation*} 
	- \Delta u +\left(I_\alpha \star |u|^p\right) |u|^{p-2}u =  |u|^{r-2} u\ \text{in } \R^N
\end{equation*}
was carried out (see also Lions \cite{MR636734} and Ruiz \cite{MR2679375}).

We look for solutions of \eqref{103} in the space $E_{rad}^{\alpha,p}(\mathbb{R}^N)$, whose definition is as follows (see \cite[Section 2]{MR3568051}). The Coulomb space $\mathcal{Q}^{\alpha,p}(\mathbb{R}^N)$ is the vector space of measurable functions $u:\mathbb{R}^N\to \mathbb{R}$ such that
\begin{equation*}
	\|u\|_{\mathcal{Q}^{\alpha,p}}=\left(\int_{\mathbb{R}^N} |I_{\alpha/2}\star |u|^p|^2dx \right)^{\frac{1}{2p}}<\infty.
\end{equation*}
The Coulomb-Sobolev space $E^{\alpha,p}(\mathbb{R}^N)$ is the subspace of functions $u\in \mathcal{Q}^{\alpha,p}(\mathbb{R}^N)$ such that $u$ is weakly differentiable in $\mathbb{R}^N$, $Du\in L^2(\mathbb{R}^N)$, and 

\begin{equation}\label{normsps}
	\|u\|=\left(\int_{\mathbb{R}^N} |\nabla u|^2dx+\left(\int_{\mathbb{R}^N} |I_{\alpha/2}\star |u|^p|^2dx\right)^{1/p}\right)^{1/2}<\infty. 
\end{equation}
The space $E^{\alpha,p}(\mathbb{R}^N)$ is a uniformly convex, Banach space with respect to the norm $\|u\|$ (see \cite[Proposition 2.2 and Proposition 2.9]{MR3568051}). We define $E_{rad}^{\alpha,p}(\mathbb{R}^N)$ to be the subspace of radial functions in $E^{\alpha,p}(\mathbb{R}^N)$. Since it is a closed subspace, it is also a uniformly convex, Banach space with respect to the norm $\|u\|$.  

We note here that  $q$ corresponds to the smaller exponent for which the continuous embedding $E^{\alpha,p}(\mathbb{R}^N)\hookrightarrow L^\ell(\mathbb{R}^N)$ holds true with $\ell\in [q,2^*]$ (see \cite[Theorem 1]{MR3568051}), while  $\tilde{q}$ corresponds to the smaller exponent for which the continuous embedding $E_{rad}^{\alpha,p}(\mathbb{R}^N)\hookrightarrow L^\ell(\mathbb{R}^N)$ with $\ell\in(\tilde{q},2^*]$ holds true (see \cite[Theorem 4]{MR3568051}).

Under these assumptions solutions of \eqref{103} are critical points of the $C^1$ functional $\Phi:E_{rad}^{\alpha,p}(\mathbb{R}^N)\to \mathbb{R}$ given by

\begin{equation}\label{functionaSPS}
	\Phi(u)=I(u)-\lambda J(u)-\mu\int_{\mathbb{R}^N} F(x,u)dx-\frac{1}{r}\int_{\mathbb{R}^N} |u|^rdx,\ u\in E_{rad}^{\alpha,p}(\mathbb{R}^N),
\end{equation}
where
\begin{equation*}
	I(u)=\frac{1}{2}\int_{\mathbb{R}^N} |\nabla u|^2dx+\frac{1}{2p}\int_{\mathbb{R}^N} |I_{\alpha/2}\star |u|^p|^2dx,\ u\in E_{rad}^{\alpha,p}(\mathbb{R}^N)
\end{equation*}
and
\begin{equation*}
	J(u)=\frac{1}{q}\int_{\mathbb{R}^N} |u|^qdx,\ u\in E_{rad}^{\alpha,p}(\mathbb{R}^N).
\end{equation*}

As we shall see the geometry of the functional $\Phi$ depends on the parameter $\lambda>0$. In fact, by introducing the scaling, that appears naturally in this context (see \cite{MR3568051}), $u_t(x)=t^\sigma u(tx)$, $u\in E_{rad}^{\alpha,p}(\mathbb{R}^N)\setminus\{0\}$, $t>0$, where
\begin{equation*}
	\sigma = \frac{2+\alpha}{2(p-1)},
\end{equation*}
one can easily see that $I(u_t)=t^sI(u)$ and $J(u_t)=t^sJ(u)$ where $s=\sigma q-N$. Therefore the behavior of $I(u)-\lambda J(u)$ near the origin will be determined by the following sequence of scaled eigenvalues introduced in \cite{MePe2}: let $\M=\{u\in E_{rad}^{\alpha,p}(\mathbb{R}^N): I(u)=1\}$, $\F$ denote the class of symmetric subsets of $\M$ and let $i(M)$ denote the $\Z_2$-cohomological index of $M \in \F$ (see Fadell and Rabinowitz \cite{MR0478189}). For $k \ge 1$, let $\F_k = \bgset{M \in \F : i(M) \ge k}$ and set
\begin{equation} \label{5}
	\lambda_k := \inf_{M \in \F_k}\, \sup_{u \in M}\,  \widetilde{\Psi}(u),
\end{equation}
where $ \widetilde{\Psi}(u)=1/J(u)$. Then as an application of Theorem \ref{Theorem 7} (see next section) we have the following
\begin{theorem} \label{thmeigen}
	$\lambda_k \nearrow \infty$ is a sequence of scaled eigenvalues of the problem $I'(u)=\lambda J'(u)$.
	\begin{enumroman}
		\item The first eigenvalue is given by
		\[
		\lambda_1 = \min_{u \in \M_s}\, \widetilde{\Psi}(u) > 0.
		\]
		\item If $\lambda_k = \dotsb = \lambda_{k+m-1} = \lambda$ and $E_\lambda$ is the set of eigenfunctions associated with $\lambda$ that lie on $\M_s$, then $i(E_\lambda) \ge m$.
		\item If $\lambda_k < \lambda < \lambda_{k+1}$, then
		\[
		i(\widetilde{\Psi}^{\lambda_k}) = i(\M_s \setminus \widetilde{\Psi}_\lambda) = i(\widetilde{\Psi}^\lambda) = i(\M_s \setminus \widetilde{\Psi}_{\lambda_{k+1}}) = k,
		\]
		where $\widetilde{\Psi}^a = \bgset{u \in \M_s : \widetilde{\Psi}(u) \le a}$ and $\widetilde{\Psi}_a = \bgset{u \in \M_s : \widetilde{\Psi}(u) \ge a}$ for $a \in \R$.
	\end{enumroman}
\end{theorem}

Theorem \ref{thmeigen} implies that when $\lambda\in (0,\lambda_1)$, the functional $I(u)-\lambda J(u)$ is coercive, and hence, if $\mu>0$ is small, then standard arguments can be used to show the existence of a local minimizer and a mountain pass type critical point. When $\lambda\ge \lambda_1$ the geometry of the functional has a drastic change and the same argument cannot be applied anymore. Our main result is the following 

\begin{theorem}\label{thmspl} Assume \eqref{f1}, \eqref{f2} and $\lambda>0$. Then, there exists $\mu^*>0$ such that, if $\mu\in(0,\mu^*)$, equation \eqref{103} has two solutions $u_1,u_2$ such that $\Phi(u_1)<0<\Phi(u_2)$.
\end{theorem}

In order to prove Theorem \ref{thmspl} we will develop a linking theorem, based on a cohomological index theory (see Theorem \ref{thmabstract}). See also Remark \ref{rmkgene} for more applications.

\subsection{Abstract Theory}

Recently in Mercuri and Perera \cite{MePe2} the following notion of scaling was introduced: let $W$ be a reflexive Banach space, then

\begin{definition}[{\cite[Definition 1.1]{MePe2}}]\label{dscaling}
	A scaling on $W$ is a continuous mapping
	\[
	W \times [0,\infty) \to W, \quad (u,t) \mapsto u_t
	\]
	satisfying
	\begin{enumerate}
		\item[$(H_1)$] $(u_{t_1})_{t_2} = u_{t_1 t_2}$ for all $u \in W$ and $t_1, t_2 \ge 0$,
		\item[$(H_2)$] $(\tau u)_t = \tau u_t$ for all $u \in W$, $\tau \in \R$, and $t \ge 0$,
		\item[$(H_3)$] $u_0 = 0$ and $u_1 = u$ for all $u \in W$,
		\item[$(H_4)$] $u_t$ is bounded on bounded sets in $W \times [0,\infty)$,
		\item[$(H_5)$] $\exists s > 0$ such that $\norm{u_t} = \O(t^s)$ as $t \to \infty$, uniformly in $u$ on bounded sets.
	\end{enumerate}
\end{definition}
As we will see this type of scaling appears naturally in many situations involving PDEs and it can be used to describe the geometry of the energy functional associated with the equation. This in turn is very helpful when trying to apply a specific critical point theorem such as, for example, the Mountain Pass Theorem or the Linking Theorem. 

Denote by $W^\ast$ the dual of $W$. Recall that $q \in C(W,W^\ast)$ is a potential operator if there is a functional $Q \in C^1(W,\R)$, called a potential for $q$, with Fr\'{e}chet derivative $Q' = q$. By replacing $Q$ with $Q - Q(0)$ if necessary, we may assume that $Q(0) = 0$.

\begin{definition}[{\cite[Definition 2.1]{MePe2}}]
	A scaled operator is an odd potential operator $\As \in C(W,W^\ast)$ that maps bounded sets into bounded sets and satisfies
	\[
	\As(u_t) v_t = t^s \As(u) v \quad \forall u, v \in W,\, t \ge 0.
	\]
\end{definition}

Let $\As$ and $\Bs$ be scaled operators satisfying
\begin{enumerate}
	\item[$(H_6)$] $\As(u) u > 0$ for all $u \in W \setminus \set{0}$,
	\item[$(H_7)$] every sequence $\seq{u_j}$ in $W$ such that $u_j \wto u$ and $\As(u_j)(u_j - u) \to 0$ has a subsequence that converges strongly to $u$,
	\item[$(H_8)$] $\Bs(u) u > 0$ for all $u \in W \setminus \set{0}$,
	\item[$(H_9)$] if $u_j \wto u$ in $W$, then $\Bs(u_j) \to \Bs(u)$ in $W^\ast$.
\end{enumerate}
The scaled eigenvalue problem
\begin{equation} \label{1}
	\As(u) = \lambda \Bs(u) \quad \text{in } W^\ast
\end{equation}
was studied in \cite{MePe2}. We say that $\lambda \in \R$ is an eigenvalue of this problem if there is a $u \in W \setminus \set{0}$, called an eigenfunction associated with $\lambda$, satisfying \eqref{1}. Then $u_t$ is also an eigenfunction associated with $\lambda$ for any $t > 0$ since
\[
\As(u_t) v = \As(u_t) (v_{t^{-1}})_t = t^s \As(u) v_{t^{-1}} = t^s \lambda \Bs(u) v_{t^{-1}} = \lambda \Bs(u_t) (v_{t^{-1}})_t = \lambda \Bs(u_t) v
\]
for all $v \in W$.

The potentials
\begin{equation} \label{2}
	I_s(u) = \int_0^1 \As(\tau u) u\, d\tau, \quad J_s(u) = \int_0^1 \Bs(\tau u) u\, d\tau, \quad u \in W
\end{equation}
of $\As$ and $\Bs$, respectively, are even, \hspace{-5pt} bounded \hspace{-5pt} on \hspace{-5pt} bounded sets, and have the scaling property
\begin{equation} \label{3}
	I_s(u_t) = t^s I_s(u), \quad J_s(u_t) = t^s J_s(u) \quad \forall u \in W,\, t \ge 0
\end{equation}
(see \cite[Proposition 2.2]{MePe2}). By \eqref{2}, $(H_6)$, and $(H_8)$,
\begin{equation} \label{4}
	I_s(u) > 0, \quad J_s(u) > 0 \quad \forall u \in W \setminus \set{0}.
\end{equation}
It follows from $(H_9)$ that if $u_j \wto u$ in $W$, then $J_s(u_j) \to J_s(u)$ (see {\cite[Proposition 2.4]{MePe2}}). We assume that $I_s$ and $J_s$ satisfy
\begin{enumerate}
	\item[$(H_{10})$] $I_s$ is coercive, i.e., $I_s(u) \to \infty$ as $\norm{u} \to \infty$,
	\item[$(H_{11})$] every solution of problem \eqref{1} satisfies $I_s(u) = \lambda\, J_s(u)$.
\end{enumerate}

We have $I_s'(0) = \As(0) = 0$ since $\As$ is odd, so the origin is a critical point of $I_s$. It is the only critical point of $I_s$ since
\[
I_s'(u) u = \As(u) u > 0 \quad \forall u \in W \setminus \set{0}
\]
by $(H_6)$. So $I_s(0) = 0$ is the only critical value of $I_s$, and hence it follows from the implicit function theorem that
\[
\M_s = \bgset{u \in W : I_s(u) = 1}
\]
is a $C^1$-Finsler manifold. Since $I_s$ is continuous, even, and coercive, $\M_s$ is complete, symmetric, and bounded. 

Let
\[
\Psi(u) = \frac{I_s(u)}{J_s(u)}, \quad u \in W \setminus \set{0}
\]
and let $\widetilde{\Psi} = \restr{\Psi}{\M_s}$. Then eigenvalues of problem \eqref{1} coincide with critical values of $\widetilde{\Psi}$ (see \cite[Proposition 2.5]{MePe2}).

Let $\F$ denote the class of symmetric subsets of $\M_s$ and let $i(M)$ denote the $\Z_2$-cohomological index of $M \in \F$ (see Fadell and Rabinowitz \cite{MR0478189}). For $k \ge 1$, let $\F_k = \bgset{M \in \F : i(M) \ge k}$ and set
\begin{equation} \label{5}
	\lambda_k := \inf_{M \in \F_k}\, \sup_{u \in M}\, \widetilde{\Psi}(u).
\end{equation}
The following theorem was proved in \cite{MePe2}.

\begin{theorem}[{\cite[Theorem 2.10]{MePe2}}] \label{Theorem 7}
	Assume $(H_1)$--$(H_{11})$. Then $\lambda_k \nearrow \infty$ is a sequence of eigenvalues of problem \eqref{1}.
	\begin{enumroman}
		\item The first eigenvalue is given by
		\[
		\lambda_1 = \min_{u \in \M_s}\, \widetilde{\Psi}(u) > 0.
		\]
		\item If $\lambda_k = \dotsb = \lambda_{k+m-1} = \lambda$ and $E_\lambda$ is the set of eigenfunctions associated with $\lambda$ that lie on $\M_s$, then $i(E_\lambda) \ge m$.
		\item If $\lambda_k < \lambda < \lambda_{k+1}$, then
		\[
		i(\widetilde{\Psi}^{\lambda_k}) = i(\M_s \setminus \widetilde{\Psi}_\lambda) = i(\widetilde{\Psi}^\lambda) = i(\M_s \setminus \widetilde{\Psi}_{\lambda_{k+1}}) = k,
		\]
		where $\widetilde{\Psi}^a = \bgset{u \in \M_s : \widetilde{\Psi}(u) \le a}$ and $\widetilde{\Psi}_a = \bgset{u \in \M_s : \widetilde{\Psi}(u) \ge a}$ for $a \in \R$.
	\end{enumroman}
\end{theorem}

We consider the question of existence and multiplicity of solutions to the nonlinear operator equation
\begin{equation} \label{18}
	\As(u) = \lambda \Bs(u) + \mu f(u) + g(u) \quad \text{in } W^\ast,
\end{equation}
where $\mu>0$, $f, g \in C(W,W^\ast)$ are potential operators satisfying
\begin{enumerate}
	\item[$(F_1)$] The potential $F$ of $f$ with $F(0)=0$ satisfies $\lim_{t\to 0^+}\frac{F(u_t)}{t^s}=\infty$ uniformly on weakly compact subsets of $W\setminus\{0\}$.
	\item[$(F_2)$] $F(u)>0$ for all $u\in W\setminus\{0\}$.
	\item[$(F_3)$] if $u_j \wto u$ in $W$, then $f(u_j) \to f(u)$ in $W^\ast$, in particular, $f$ map bounded sets into bounded sets.
	\item[$(G_1)$] The potential $G$ of $g$ with $G(0)=0$ satisfies $G(u_t)=o(t^\eta)$ as $t\to 0^+$, $\eta>s$, uniformly in $u$ on bounded sets.
	\item[$(G_2)$] $G(u)>0$ for all $u\in W\setminus\{0\}$.
	\item[$(G_3)$] $\lim_{t\to \infty}\frac{G(u_t)}{t^s}=\infty$ uniformly on weakly compact subsets of $W\setminus\{0\}$.
		\item[$(G_4)$] $G$ is bounded on bounded subsets of $W$.
\end{enumerate}

Solutions of equation \eqref{18} coincide with critical points of the $C^1$-functional
\[
\Phi_\lambda(u) = I_s(u) - \lambda\, J_s(u) - \mu F(u) - G(u), \quad u \in W,
\]
where
\[
F(u) = \int_0^1 f(\tau u) u\, d\tau, \quad G(u) = \int_0^1 g(\tau u) u\, d\tau.
\]
It follows from $(F_3)$ that if $u_j \wto u$ in $W$, then $F(u_j) \to F(u)$ (see {\cite[Proposition 2.4]{MePe2}}). As a consequence of Theorem \ref{Theorem 7}, the geometry of the functional $\Phi_\lambda$ will change according to the value of $\lambda>0$. In particular, when $\lambda\ge \lambda_1$ the functional $\Phi_\lambda$ does not possess the mountain pass geometry. For this reason we will develop a linking theorem, based on a cohomological index theory, by using the scaling $u_t$. To this end consider also 

\begin{enumerate}
	\item[$(H_{12})$] for all $u\in W\setminus\{0\}$, the functions $(0,\infty)\to \mathbb{R}$, $t\mapsto I_s(tu), J_s(tu)$ are strictly increasing.
	\item[$(H_{13})$] the function $I_s$ is weakly sequentially  lower semi continuous.
\end{enumerate}

Condition $(H_{12})$ combined with $(H_{10})$, $(H_{6})$ and $I_s(0)=0$ implies that, for each $u \in W \setminus \set{0}$, there exists a unique $t(u)>0$ such that $I_s(t(u)u)=1$. We write $\pi_{rad}:W \setminus \set{0} \to \mathcal{M}_s$ to denote the radial projection $\pi_{rad}(u)=t(u)u$ (see Proposition \ref{radialprojection}).

\begin{theorem}\label{thmlinking} Assume $(H_1),(H_3)$, $(H_6)$, $(H_{10})$ and $(H_{12})$. Let $E\in C^1(W)$, $A_0,B_0\subset \mathcal{M}_s$ such that $A_0\cap B_0=\emptyset$ and
	\begin{equation*}
		i(A_0)=i(\mathcal{M}_s\setminus B_0)=k<\infty.
	\end{equation*}
	Assume that there exist $\tilde{w}\in \mathcal{M}_s\setminus A_0$, $0\le r<\rho<R$ and $a<b$ such that, setting
	\begin{equation*}
		A_1=\{\pi_{rad}((1-\tau)v+\tau \tilde{w}):v\in A_0, 0\le \tau\le 1\},
	\end{equation*}
	\begin{equation*}
		A^*=\{u_t:u\in A_1, r\le t\le R\},
	\end{equation*}
	\begin{equation*}
		B^*=\{w_t: w\in B_0, 0\le t\le \rho\},
	\end{equation*}
	\begin{equation*}
		A=\{u_r:u\in A_1\}\cup \{v_t:v\in A_0,r\le t\le R\}\cup\{u_R:u\in A_1\},
	\end{equation*}
	\begin{equation*}
		B=\{w_{\rho}:w\in B_0\},
	\end{equation*}
	we have that 
	\begin{equation*}
		a<\inf _{B^*} E,\ \ \ \sup_AE<\inf _B E, \ \ \ \sup_{A^*} E<b.
	\end{equation*} 
	If $E$ satisfies the {\em \PS{}} condition for all $c\in (a,b)$, then $E$ has a pair of critical points $u_1,u_2$ with 
	\begin{equation*}
		\inf _{B^*}E\le E(u_1)\le \sup_A E, \ \ \ \inf _BE\le E(u_2)\le \sup_{A^*} E.
	\end{equation*}
\end{theorem}

We will show that $\Phi_\lambda$ satisfies the geometry of Theorem \ref{thmlinking}.

The main result concerning equation \eqref{18} is

\begin{theorem}\label{thmabstract} Assume $(H_1)-(H_{13})$, $(F_1)-(F_3)$, $(G_1)-(G_3)$, $\lambda>0$ and that $\Phi_\lambda$ satisfies the {\em \PS{}} condition. Then, there exists $\mu^*>0$ such that, if $\mu\in(0,\mu^*)$, equation \eqref{18} has two solutions $u_1,u_2$ such that $\Phi_\lambda(u_1)<0<\Phi_\lambda(u_2)$.
\end{theorem}

\begin{remark}\label{rmkgene} Theorem \ref{thmabstract} is a generalization of \cite[Theorem 1.2]{Pe} under the more general scaling $u_t$. Indeed, all applications in \cite{Pe} concerning subcritical equations can also be deduced from Theorem \ref{thmabstract}. It is sufficient to consider the scaling $u_t=tu$. Therefore, as stated in the introduction, the theory can be applied, for example, to equations involving the $p$-Laplacian, Kirchhoff, or fractional $p$-Laplacian operators.
	
\end{remark}

\section{Proof of Theorem \ref{thmabstract}}

This section is devoted to the proof of Theorem \ref{thmabstract}. We start with some useful properties concerning the radial projection.

\begin{proposition}\label{radialprojection} Assume $(H_6)$, $(H_{10})$ and $(H_{12})$. Then, for each $u \in W \setminus \set{0}$, there exists a unique $t(u)>0$ such that $I_s(t(u)u)=1$. Moreover, the mapping $W \setminus \set{0} \to \M_s,\, u \mapsto \pi_{rad}(u):=t(u)u$ is continuous.
\end{proposition}
\begin{proof} Since $I_s(0)=0$ and $\lim_{t\to \infty}I_s(tu)=\infty$ by $(H_{10})$ it follows by $(H_6)$ and $(H_{12})$ that there exists a unique $t(u)>0$ such that $I_s(t(u)u)=1$. We claim that $W \setminus \set{0} \to \mathbb{R},\, u \mapsto t(u)$ is continuous, which implies that $\pi_{rad}$ is continuous. Indeed, suppose that $u_n\to u\in W\setminus\{0\}$. Since $I_s(0)=0$ it follows that $t(u_n)$ does not converge to $0$, then, by $(H_{10})$ we can assume that $t(u_n)\to t>0$ and hence $I_s(tu)=\lim_{n\to \infty}I_s(t(u_n)u_n)=1$ which implies, by uniqueness, that $t=t(u)$, and the proof is complete.
	
\end{proof}
There is also a natural projection associated with the scaling defined in Definition \ref{dscaling}. For $u \in W \setminus \set{0}$, since \eqref{4} holds, we can set
\begin{equation} \label{1000}
	t_u = I_s(u)^{-1/s}, \quad \pi(u) = u_{t_u}.
\end{equation}

\begin{proposition}\label{pprojection} Assume $(H_1),(H_3)$ and $(H_6)$. The mapping $W \setminus \set{0} \to \M_s,\, u \mapsto \pi(u)$ is continuous. Moreover $\pi_{|\M_s}$ is the identity map and, if $\pi(u)=\pi(v)$, then there exists $t>0$ such that $u=v_t$.
\end{proposition}
\begin{proof} Condition $(H_6)$ guarantees that $\pi$ is well define (see \eqref{4}). Since $I_s$ is continuous, it is clear that $\pi$ is also continuous as the scaling is continuous. Note that if $u\in \mathcal{M}_s$, then $t_u=I_s(u)^{-1/s}=1$, and hence, by $(H_3)$, we conclude that $\pi(u)=u_1=u$, which implies that $\pi_{|\M_s}$ is the identity map. To conclude the proof suppose that $w:=u_{t_u}=v_{t_v}$. Then $w_{t_u^{-1}}=(u_{t_u})_{t_u^{-1}} $ and hence, by $(H_1)$, it follows that $w_{t_u^{-1}}=u$, which implies that $v_{t_vt_u^{-1}}=u$, and the proof is complete.
\end{proof}

Before proving Theorem \ref{thmlinking} we need some technical results.

\begin{lemma}\label{defa2} Assume $(H_1),(H_3)$, and $(H_6)$. Let $A_0,B_0\subset \mathcal{M}_s$ be such that $A_0\cap B_0=\emptyset$ and
	\begin{equation*}
		i(A_0)=i(\mathcal{M}_s\setminus B_0)=k<\infty.
	\end{equation*}
	Fix $R>0$ and set 
	\begin{equation*}
		D_1=\{v_R:v\in A_0\},\ D_2=\{w_t:w\in B_0, t\ge 0\}.
	\end{equation*}
	Then $D_1$ links $D_2$ cohomologically in dimension $k-1$, that is, the homomorphism $\iota^\ast:	\widetilde{H}^{k-1}(W\setminus D_2)\to 	\widetilde{H}^{k-1}(D_1)$, induced by the inclusion $D_1 \hookrightarrow W\setminus D_2$ is non trivial.
\end{lemma}
\begin{proof} 	Consider the commutative diagram 
	\[
	\begin{tikzcd}
		\widetilde{H}^{k-1}(\mathcal{M}_s\setminus B_0) \arrow[d, "\approx"] \arrow[r, "j^\ast"'] & 	\widetilde{H}^{k-1}(A_0) \arrow[d, "\approx"'] \\
		\widetilde{H}^{k-1}(W\setminus D_2)\arrow[r, "\iota^\ast"'] & 	\widetilde{H}^{k-1}(D_1) 
	\end{tikzcd}
	\]
	where $j: A_0\to \mathcal{M}_s\setminus B_0$ is the inclusion map and the vertical lines are isomorphisms induced by the homotopy equivalence $\pi_{|W\setminus D_2}: W\setminus D_2\to \mathcal{M}_s\setminus B_0$ and the homeomorphism $\pi_{|D_1}: D_1\to  A_0$. Indeed, $\pi_{|D_1}$ is a homeomorphism as a straightforward consequence of Proposition \ref{pprojection}.
	 
	  Let us show that $\pi_{|W\setminus D_2}$ is in fact a homotopy equivalence. By Proposition \ref{pprojection} we know that $\pi_{|\M_s}$ is the identity map and hence $\pi_{|W\setminus D_2}\circ \pi_{|\M_s}=\pi_{|\M_s\setminus B_0}:\M_s\setminus B_0\to \M_s\setminus B_0$ is the identity map. Now we claim that $ \pi_{|\M_s}\circ \pi_{|W\setminus D_2}:W\setminus D_2\to W\setminus D_2$ is homotopic to the identity map. Indeed, the map $W\setminus D_2\times [0,1]\to W\setminus D_2$, $(u,t)\mapsto$ $u_{(1-t)t_u+t}$ is the desired homotopy.
	
	Since $j^\ast\neq 0$ by \cite[Proposition 2.14 item (iv)]{MR2640827}, the proof is complete.
\end{proof}

\begin{theorem}\label{HomLin} Assume $(H_1),(H_3)$, $(H_6)$, $(H_{10})$ and $(H_{12})$. Let $A_0,B_0\subset \mathcal{M}_s$ be such that $A_0\cap B_0=\emptyset$ and
	\begin{equation*}
		i(A_0)=i(\mathcal{M}_s\setminus B_0)=k<\infty.
	\end{equation*}
	Let $\tilde{w}\in \mathcal{M}\setminus A_0$ and $0\le r<\rho<R$. Set $A_1=\{\pi_{rad}((1-\tau)v+\tau \tilde{w}):v\in A_0, 0\le \tau\le 1\}$ and
	\begin{equation*}
		A=\{u_r:u\in A_1\}\cup \{v_t:v\in A_0,r\le t\le R\}\cup\{u_R:u\in A_1\},\ B=\{w_{\rho}:w\in B_0\}.
	\end{equation*}
	Then $A$ links $B$ cohomologically in dimension $k$.
\end{theorem}
\begin{proof}  Write 
	\begin{equation*}
		A_2=\{u_r:u\in A_1\}\cup \{v_t:v\in A_0, r\le t\le R\},\ A_3=\{v_R:v\in A_0\},
	\end{equation*}
	and
	\begin{equation*}
		B_1=\{w_t:w\in B_0, t\ge 0\},\ B_2=\{w_t:w\in B_0, t\ge \rho\}.
	\end{equation*}
	Consider the commutative diagram 
	\[
	\begin{tikzcd}
		\widetilde{H}^{k-1}(W\setminus B_2) \arrow[r, "{}"] \arrow[d, "{}"'] & \widetilde{H}^{k-1}(W\setminus B_1) \arrow[d, "\iota^\ast_1"] \arrow[r, "{}"] & H^{k}(W\setminus B_2,W\setminus B_1) \arrow[d, "\iota^\ast_2"] \arrow[r, "{}"] & \widetilde{H}^{k}(W\setminus B_2) \arrow[d, "{}"] \\
		\widetilde{H}^{k-1}(A_2) \arrow[r, "{}"']                            & \widetilde{H}^{k-1}(A_3) \arrow[r, "\delta"']                                 & H^{k}(A_2,A_3) \arrow[r, "{}"']                                                & \widetilde{H}^{k}(A_2)
	\end{tikzcd}.
	\]
	We first prove that $A_2$ is contractible. Indeed, we claim that the map $A_2\times [0,1]\to A_2$, $(u,t)\mapsto u_{1-t+rtI_s(u)^{-1/s}}$ is a strong deformation retraction of $A_2$ onto $\{u_r:u\in A_1\}$. In fact, the map is continuous because the scaling is continuous. If $w=u_r$ with $u\in A_1$, then $I_s(w)=r^s$ and hence $w_{1-t+rtI_s(w)^{-1/s}}=w_{1}=w\in A_2$. If $u=v_\tau$ with $v\in A_0$ and $r\le \tau\le R$, then $(v_\tau)_{1-t+rtI_s(v_\tau)^{-1/s}}=v_{(1-t)\tau+rt}$ and clearly $r\le (1-t)\tau+rt \le R$ for all $t\in [0,1]$, so $(v_\tau)_{1-t+rtI_s(v_\tau)^{-1/s}}\in A_2$. To conclude observe that $w_{rI_s(w)^{-1/s}}=w$ if $w=u_r$ with $u\in A_1$, and  $(v_\tau)_{rI_s(v_\tau)^{-1/s}}=v_{r}$ if $v\in A_0$ and $\tau \in[r,R]$. 
	
	In the previous paragraph we have show that $A_2$ strongly retracts to $\{u_r:u\in A_1\}$, so to complete the proof that $A_2$ is contractible, it is sufficient to prove that $\{u_r:u\in A_1\}$ is contractible. Since $\{u_r:u\in A_1\}$ is clearly homeomorphic to $A_1$, it is sufficient to prove that $A_1$ is contractible. Since $\pi_{rad}$ is continuous, by Proposition \ref{radialprojection}, one can easily see that the map $A_1\times [0,1]\to A_1$, $(u,t)\mapsto \pi_{rad}((1-t)u+t\tilde{w})$ is a contraction of $A_1$ to $\tilde{w}$.  
	
	By Lemma \ref{defa2}, we have that $\iota_1^\ast\neq 0$. Since $\widetilde{H}^{\ast}(A_2)=0$, we conclude that $\delta$ is an isomorphism by exactness of the bottom row. Therefore $\iota_2^\ast\neq 0$. Now let
	\begin{equation*}
		A_4=\{u_R:u\in A_1\},
	\end{equation*}
	and consider the commutative diagram 
	\[
	\begin{tikzcd}
		H^{k}(W\setminus B,W\setminus B^*) \arrow[r, "\approx"] \arrow[d, "\iota_3^\ast"'] & 	H^{k}(W\setminus B_2,W\setminus B_1) \arrow[d, "\iota_2^\ast"'] \\
		H^{k}(A,A_4)\arrow[r, "{}"'] & 	H^{k}(A_2,A_3) 
	\end{tikzcd}
	\]
	induced by inclusions, where $B^*=\{w_t: w\in B_0, 0\le t\le \rho\}$. The top arrow is an isomorphism by the excision property since $\{w_t: w\in B_0, t>\rho\}$ is a closed subset of $W\setminus B$ which is contained in the open subset $W\setminus B^*$. It follows that $\iota_3^*\neq 0$. 
	
	Finally consider the commutative diagram 
	\[
	\begin{tikzcd}
		\widetilde{H}^{k-1}(W\setminus B^*) \arrow[r, "{}"] \arrow[d, "{}"'] & 	H^{k}(W\setminus B,W\setminus B^*) \arrow[d, "\iota^\ast_3"] \arrow[r, "{}"] & 	\widetilde{H}^{k}(W\setminus B) \arrow[d, "\iota^\ast"] \arrow[r, "{}"] & 	\widetilde{H}^{k}(W\setminus B^*) \arrow[d, "{}"] \\
		\widetilde{H}^{k-1}(A_4) \arrow[r, "{}"'] & 	H^{k}(A,A_4) \arrow[r, "j^*"'] & \widetilde{H}^{k}(A) \arrow[r, "{}"'] & \widetilde{H}^{k}(A_4)
	\end{tikzcd}.
	\]
	Since $A_4$ is homeomorphic to $A_1$ (by Proposition \ref{pprojection}), and hence contractible, $\widetilde{H}^{*}(A_4)=0$, so $j^*$ is an isomorphism. Since $\iota_3^*\neq 0$, it follows that $\iota^*\neq 0$, and the proof is complete.
\end{proof}

Now we are in a position to prove Theorem \ref{thmlinking}:

\begin{proof}[Proof of Theorem \ref{thmlinking}]
 Since $B^*\cap A$ and $B\cap A^*$ are nonempty we have $\inf_{B^*} E\le \sup_A E$ and $\inf _B E\le \sup _{A^*}E$. Let us prove that if $\alpha<\beta<\gamma$ satisfy 
	\begin{equation*}
		a<\alpha<\inf_{B^*} , \ \ \ \sup_A E<\beta<\inf_B E, \ \ \ \sup_{A^*} E<\gamma<b,
	\end{equation*}
	then
	\begin{equation*}
		H^k(E^\beta,E^\alpha)\neq 0, \ \ \ H^{k+1}(E^\gamma,E^\beta)\ne0 .
	\end{equation*}
	By Theorem \ref{HomLin}, $A$ links $B$ cohomologically in dimension $k$ and hence the inclusion $\iota:A \hookrightarrow W\setminus B$ induces a nontrivial homomorphism. Since $E<\beta$ in $A$ and $E>\beta$ on $B$, we also have the inclusion $\iota_1:A \hookrightarrow E^\beta$ and $\iota_2:E^\beta \hookrightarrow W\setminus B$, which induce the commutative diagram 
	\[
	\begin{tikzcd}
		\widetilde{H}^{k}(W\setminus B) \arrow[d, "\iota_2^*"] \arrow[rd, "\iota^*"]  \\
		\widetilde{H}^{k}(E^\beta)\arrow[r, "\iota_1^*"'] & 	\widetilde{H}^{k}(A). 
	\end{tikzcd}
	\]
	Therefore $\iota_1^*\iota_2^*=\iota^*\neq 0$ and hence both $\iota_1^*,\iota_2^*$ are nontrivial.
	
	First we show that $H^k(E^\beta,E^\alpha)\neq 0$. Since $E>\alpha$ in $B^*$ and $\alpha<\beta$, we have the inclusions $E^\alpha \hookrightarrow W\setminus B^*\hookrightarrow W\setminus B$ and $E^\alpha \hookrightarrow E^\beta\hookrightarrow W\setminus B$, yielding the commutative diagram 
	\[
	\begin{tikzcd}
		\widetilde{H}^{k}(W\setminus B) \arrow[d, "\iota_2^*"] \arrow[r, "{}"] &  	\widetilde{H}^{k}(W\setminus B^*) \arrow[d, "{}"] \\
		\arrow[r, "\iota_3^*"']\widetilde{H}^{k}(E^\beta) & 	\widetilde{H}^{k}(E^\alpha). 
	\end{tikzcd}
	\]
	We claim that the set $W\setminus B^*$ is contractible. Indeed, for any $\rho'>\rho$, the mapping
	\begin{equation*}
		(W\setminus B^*)\times [0,1]\to W\setminus B^*,\ (u,t)\mapsto u_{1-t+t\rho'I_s(u)^{-1/s}}
	\end{equation*}
	is a strong deformation retraction of $W\setminus B^*$ onto $\{\rho' u:u\in \mathcal{M}_s\}$, which is a strong deformation retract of $W\setminus\{0\}$ (by Proposition \ref{pprojection}) and hence contractible. In fact, the proof is very similar to the one used in Theorem \ref{HomLin} to prove that $A_2$ is contractible. It follows that $\widetilde{H}^*(W\setminus B^*)=0$ and hence 
	\begin{equation*}
		\iota_3^*\iota_2^*=0.
	\end{equation*}
	Therefore $\iota_3^*$ is not injective. Now consider the exact sequence of the pair $(E^\beta,E^\alpha)$:
	\[
	\begin{tikzcd}
		\cdots \arrow[r, "{}"]	& H^{k}(E^\beta,E^\alpha) \arrow[r, "j^*"]  &  	\widetilde{H}^{k}(E^\beta) \arrow[r, "i_3^*"] &\widetilde{H}^{k}(E^\alpha) \arrow[r, "\delta"] & \cdots.
	\end{tikzcd}
	\]
	By exactness,
	\begin{equation*}
		\operatorname{im} j^*=\operatorname{ker} \iota_3^*\neq 0,
	\end{equation*}
	so $H^k(E^\beta,E^\alpha)\neq 0$ as desired. 
	
	Now we claim that $H^{k+1}(E^\gamma,E^\beta)\neq 0$.  Since $E<\gamma$ on $A^*$ and $\beta<\gamma$, we have the inclusions $A \hookrightarrow A^*\hookrightarrow E^\gamma$ and $A \hookrightarrow E^\beta\hookrightarrow E^\gamma$, yielding the commutative diagram 
	\[
	\begin{tikzcd}
		\widetilde{H}^{k}(E^\gamma) \arrow[d, "\iota_4^*"] \arrow[r, "{}"] &  	\widetilde{H}^{k}(A^*) \arrow[d, "{}"] \\
		\arrow[r, "\iota_1^*"']\widetilde{H}^{k}(E^\beta) & 	\widetilde{H}^{k}(A). 
	\end{tikzcd}
	\]
	We claim that $A^*$ is contractible. Indeed, the mapping 
	
	\begin{equation*}
		A^*\times [0,1]\to A^*,\ (u,t)\mapsto u_{1-t+rtI_s(u)^{-1/s}}
	\end{equation*}
	is a strong deformation retraction of $A^*$ onto $\{u_r:u\in A_1\}$, which is homeomorphic to $A_1$ (by Proposition \ref{pprojection}) and hence contractible as in the proof of Theorem \ref{HomLin}. It follows that $\widetilde{H}^*(A^*)=0$ and hence
	\begin{equation*}
		\iota_1^*\iota_4^*=0.
	\end{equation*}
	Therefore $\iota_4^*$ is not surjective. Now consider the exact sequence of the pair $(E^\gamma,E^\beta)$:
	\[
	\begin{tikzcd}
		\cdots \arrow[r, "j^*"]	& H^{k}(E^\gamma) \arrow[r, "\iota_4^*"]  &  	\widetilde{H}^{k}(E^\beta) \arrow[r, "\delta"] &\widetilde{H}^{k+1}(E^\gamma,E^\beta) \arrow[r, "j^*"] & \cdots.
	\end{tikzcd}
	\]
	By exactness,
	\begin{equation*}
		\operatorname{ker} \delta=\operatorname{im} \iota_4^*\neq \widetilde{H}^{k}(E^\beta),
	\end{equation*}
	so $H^k(E^\gamma,E^\beta)\neq 0$ as desired. 
\end{proof}

The next result will ensure the proper choice of $A_0$ and $B_0$ that we have to make in Theorem \ref{thmlinking} in order to apply it in the proof of Theorem \ref{thmabstract}.

\begin{lemma}\label{wcs}Assume $(H_1)-(H_{13})$. Let  $A_0=\widetilde{\Psi}^{\lambda_k}$, $\tilde{w}\in \mathcal{M}_s\setminus A_0$, and $A_1=\{\pi_{rad}((1-\tau)v+\tau \tilde{w}):v\in A_0, 0\le \tau\le 1\}$. Then there exists $\varepsilon>0$ such that $J_s(u)>\varepsilon$ for all $u\in A_1$. Moreover $A_1$ is a weakly compact subset of $W\setminus \{0\}$. 
\end{lemma}
\begin{proof} Indeed, given $v\in A_0$ and $\tau\in [0,1]$, write $w_{v,\tau}=(1-\tau)v+\tau \tilde{w}$. Suppose, on the contrary, that there exist sequences $v_n\in A_0$ and $\tau_n\in [0,1]$ such that $J_s(\pi_{rad}(w_{v_n,\tau_n}))\to 0$.  Since $A_0\subset \mathcal{M}_s$ is bounded we obtain from  $I_s(t(w_{v_n,\tau_n})w_{v_n,\tau_n})=1$ that there exists $C>0$ such that $t(w_{v_n,\tau_n})\ge C$ for all $n\in \mathbb{N}$, therefore, by $(H_{12})$ we have
	\begin{equation*}
	0\le\lim_{n\to \infty}J_s(Cw_{v_n,\tau_n})\le\lim_{n\to \infty} J_s(t(w_{v_n,\tau_n})w_{v_n,\tau_n})=\lim_{n\to \infty}J_s(\pi_{rad}(w_{v_n,\tau_n}))=0,
	\end{equation*}
	therefore $\lim_{n\to \infty}J_s(Cw_{v_n,\tau_n})=0$. We can assume without loss of generality that $v_n \rightharpoonup v$, $\tau_n\to \tau $. Note that $v\neq 0$ since, if not, then $\lambda_k\ge \widetilde{\Psi}(v_n)\to \infty$, a contradiction. We have $w_{v_n,\tau_n}\rightharpoonup (1-\tau)v+\tau \tilde{w}:=w$.  Then $J_s(Cw_{v_n,\tau_n})\to J_s(Cw)=0$, which implies, by \eqref{4}, that $w=0$. 
	
	If $\tau=0$, then $v=0$, which is a contradiction. If $\tau=1$, we clearly have the contradiction $\tilde{w}=0$, so $\tau\in (0,1)$. Since $(1-\tau)v=-\tau \tilde{w}$, then $v=c \tilde{w}$ where $c=-\tau/(1-\tau)$. We claim that $|c|\le 1$. Indeed, if $|c|>1$, then by $(H_{12})$ and $(H_{13})$, we have 
	\begin{equation*}
		1=I_s(\tilde{w})<	I_s(|c| \tilde{w})=	I_s(c \tilde{w})=I_s(v)\le \liminf I_s(v_n)= 1, 
	\end{equation*}
	which is a contradiction. Therefore, by $(H_{12})$, we conclude that
	\begin{equation*}
		\lambda_k\ge  \widetilde{\Psi}(v_n)=\frac{1}{J_s(v_n)}\to \frac{1}{J_s(v)}=\frac{1}{J_s(|c|\tilde{w})}\ge \frac{1}{J_s(\tilde{w})}=\widetilde{\Psi}(\tilde{w})>\lambda_k.
	\end{equation*}
	We have reached a contradiction. To conclude note that if $u_n\in A_1$ satisfies $u_n \rightharpoonup 0$, then $J(u_n)\to 0$, which contradicts the inequality just proved. 
\end{proof}

 Now we can prove Theorem \ref{thmabstract}.

\begin{proof}[Proof of Theorem \ref{thmabstract}] If $\lambda\in (0,\lambda_1)$, take $A_0=\emptyset$ and $B_0=\mathcal{M}_s$. If $\lambda\ge \lambda_1$, choose $k$ such that $\lambda\in[\lambda_k,\lambda_{k+1})$. Let $A_0=\widetilde{\Psi}^{\lambda_k}$ and $B_0=\widetilde{\Psi}_{\lambda_{k+1}}$. We have
	   \[
	i(\widetilde{\Psi}^{\lambda_k}) = i(\M_s \setminus \widetilde{\Psi}_{\lambda_{k+1}}) = k
	\]
	by Theorem \ref{Theorem 7}. For $u\in \mathcal{M}_s$ and $t>0$, we have 
	\begin{equation}\label{z1}
		\Phi_\lambda(u_t)=t^s\left(1-\frac{\lambda}{\widetilde{\Psi}(u)} \right)-\mu F(u_t)-G(u_t).
	\end{equation}
	Take $u\in B_0$. Then by \eqref{z1} and $(G_1)$, 
	\begin{equation*}
		\Phi_\lambda(u_t)\ge t^s\left(1-\frac{\lambda}{\lambda_{k+1}} +o(1)\right)-\mu F(u_t)\ \mbox{as}\ t\to 0^+.
	\end{equation*}
	Now given $\varepsilon>0$, by using $(F_3)$, we can find $\rho,\mu^*>0$ such that $\Phi_\lambda(u_{\rho})>\varepsilon$ for all $\mu\in (0,\mu^*)$ and hence
	\begin{equation*}
		\inf _{B}\Phi_\lambda>0.
	\end{equation*}
	Fix $0<\mu<\mu^*$, let $\tilde{w}\in \mathcal{M}_s\setminus A_0$, and let 
	\begin{equation*}
		A_1=\{\pi_{rad}((1-\tau)v+\tau \tilde{w}):v\in A_0, 0\le \tau\le 1\}.
	\end{equation*}
	 If $u\in A_1$, then by \eqref{z1} and $(G_2)$ we obtain
	\begin{equation*}
		\Phi_\lambda (u_t)\le t^s\left(1-\mu\frac{F(u_t)}{t^s}\right).
	\end{equation*}
	Note, by Lemma \ref{wcs}, that $A_1$ is a weakly compact subset of $W\setminus \{0\}$, which implies by the above inequality and $(F_1)$ that there exists $0<r<\rho$ such that
	\begin{equation}\label{z2}
		\sup\{\Phi_\lambda (u_r):u\in A_1\}<0.
	\end{equation}
	Similarly,
	\begin{equation*}
		\Phi_\lambda (u_t)\le t^s\left(1-\mu\frac{G(u_t)}{t^s}\right)
	\end{equation*}
	for $u\in A_1$ and hence it follows from $(G_3)$ that there exists $R>\rho$ such that 
	\begin{equation}\label{z3}
		\sup\{\Phi_\lambda (u_R):u\in A_1\}<0.
	\end{equation}
	To conclude, consider the set $ D=\{v_t:v\in A_0,r\le t\le R\}$. We claim that $D$ is a weakly compact subset of $W\setminus \{0\}$. Indeed, by Lemma \ref{wcs} we know that there exists $\varepsilon>0$ such that $J_s(v)> \varepsilon>0$ for all $v\in A_0$. Therefore $J_s(v_t)=t^sJ_s(v)>r^s\varepsilon$ for all $v\in D$ and hence the claim is proved. Given $v\in A_0$ note that
	\begin{equation*}
		\Phi_\lambda (v_t)\le t^s\left(1-\frac{\lambda}{\widetilde{\Psi}(v)}\right)-\mu F(v_t)-G(v_t).
	\end{equation*}
	Since $D$ is a weakly compact subset of $W\setminus \{0\}$, it follows that there exists $\delta>0$ such that $-\mu F(v_t)<-\delta$. Indeed, if not, we can find a sequence $w_n\in D$ such that  $-\mu F(w_n)\to 0$. Since $D$ is bounded, then $w_n \rightharpoonup w$ and hence  $-\mu F(w)=0$ implying that $w=0$, a contradiction. Therefore
	
	\begin{equation*}
		\Phi_\lambda (v_t)\le t^s\left(1-\frac{\lambda}{\widetilde{\Psi}(v)}\right)-\delta, v\in A_0, t\in [r,R].
	\end{equation*}
	Thus
	\begin{equation}\label{z4}
		\sup\{\Phi_\lambda (v_t):v\in A_0, t\in[r,R]\}<0.
	\end{equation}
	To finish the proof we need to prove that there exist $a<b$ such that 
	\begin{equation*}
			a<\inf _{B^*} \Phi_\lambda\ \ \ \mbox{and}\ \ \   \sup_{A^*} \Phi_\lambda<b.
	\end{equation*}
Since $A^*$ and $B^*$ are bounded sets, these inequalities are clear from $(H_9)$, $(H_{10})$, $(F_3)$, and $(G_4)$. Combining these inequalities with \eqref{z2}, \eqref{z3}, \eqref{z4}, and Theorem \ref{thmlinking}, the proof is complete.
\end{proof}

\section{Proof of Theorem \ref{thmspl}} 

In this section we prove Theorem \ref{thmspl}. The following straightforward inequality will be useful

\begin{equation}\label{ineqnorms}
	\min\set{t^{\frac{s}{2}},t^{\frac{s}{2p}}} \norm{u}\le\norm{u_t} \le \max \set{t^{\frac{s}{2}},t^{\frac{s}{2p}}} \norm{u}, u\in E_{rad}^{\alpha,p}(\mathbb{R}^N)\setminus\{0\}, t>0,
\end{equation}
where $\|u\|$ was defined at \eqref{normsps} and $s=\sigma q-N$.

\begin{lemma}\label{PS} The functional $\Phi$ satisfies the {\em \PS{}} condition.
\end{lemma}
\begin{proof} Let $u_n\in E_{rad}^{\alpha,p}(\mathbb{R}^N)$ be a sequence such that $\Phi(u_n)$ is bounded and $\Phi'(u_n)\to 0$ as $n\to \infty$. We claim that $u_n$ is bounded. In fact, suppose on the contrary, that $\|u_n\|\to \infty$ as $n\to \infty$ and write $v_n=(u_n)_{I(u_n)^{\frac{-1}{s}}}$. Clearly $I(v_n)=1$ and $v_n$ is bounded (by \eqref{normsps} and \eqref{ineqnorms}).  It follows that
	\begin{equation}\label{eqps0}
		1-\frac{\lambda}{q} J(v_n)-\mu\int_{\mathbb{R}^N} \frac{ F(u_n)}{I(u_n)}dx-\frac{I(u_n)^{\frac{\sigma (r-q)}{s}}}{r}\int_{\mathbb{R}^N} |v_n|^rdx=o(1),\ n\in \mathbb{N},
	\end{equation}
	and
		\begin{equation}\label{eqps}
	\int_{\mathbb{R}^N} |\nabla v_n|^2dx+\int_{\mathbb{R}^N} |I_{\alpha/2}\star |v_n|^p|^2dx-\lambda J(v_n)-\mu \int_{\mathbb{R}^N}\frac{ f(u_n)u_n}{I(u_n)}dx-I(u_n)^{\frac{\sigma (r-q)}{s}}\int_{\mathbb{R}^N} |v_n|^rdx=o(1).
	\end{equation}
	Now we claim that $\int_{\mathbb{R}^N}\frac{ F(u_n)}{I(u_n)}dx,\int_{\mathbb{R}^N}\frac{ f(u_n)u_n}{I(u_n)}dx\to 0$ as $n\to \infty$. We prove only one of the limits since the other is similar. Indeed, note by \eqref{f1} that
	\begin{eqnarray*}
		\left|\int_{\mathbb{R}^N}\frac{ f(u_n)u_n}{I(u_n)}\right|dx&\le& I(u_n)^{\frac{\sigma (\ell-q)}{s}}\int_{\mathbb{R}^N} \frac{|f( I(u_n)^{\frac{\sigma }{s}}v_n) I(u_n)^{\frac{\sigma }{s}}v_n|}{ I(u_n)^{\frac{\sigma \ell}{s}}}dx \\
		&\le& CI(u_n)^{\frac{\sigma (\ell-q)}{s}}\int_{\mathbb{R}^N} |v_n|^\ell dx, \ \forall n\in \mathbb{N}.
	\end{eqnarray*}
	Since $\ell<q$ and $\int_{\mathbb{R}^N} |v_n|^\ell dx$ is bounded by \cite[Theorem 4]{MR3568051}, it follows that $\int_{\mathbb{R}^N}\frac{ f(u_n)u_n}{I(u_n)}dx\to 0$ as $n\to \infty$. Therefore we conclude from \eqref{eqps0} that $\int_{\mathbb{R}^N} |u_n|^rdx\to 0$ and since $u_n \rightharpoonup u$ in $E_{rad}^{\alpha,p}(\mathbb{R}^N)$, we obtain that $u=0$  hence $J(v_n)\to 0$ and $\int_{\mathbb{R}^N} |I_{\alpha/2}\star |v_n|^p|^2dx\to 0$.
	
	To proceed, multiply \eqref{eqps0} by $-r$ and add it to \eqref{eqps} to obtain
	\begin{equation*}
			\int_{\mathbb{R}^N} |\nabla v_n|^2dx-r=o(1),\ n\in \mathbb{N},
	\end{equation*}
	however, by $I(v_n)=1$ we also have that
	\begin{equation*}
		\int_{\mathbb{R}^N} |\nabla v_n|^2dx-2=o(1),\ n\in \mathbb{N},
	\end{equation*}
	and since $r>q>2$ we get a contradiction. Therefore $u_n$ is bounded. The proof now follows from \cite[Lemma 2.10]{PeSi0}.
	
\end{proof}

\begin{proof}[Proof of Theorem \ref{thmspl}]

 Let us verify that all conditions of Theorem \ref{thmabstract} are satisfied. Indeed, recall the scaling defined at the introduction: $u_t(x)=t^\sigma u(tx)$, $u\in E_{rad}^{\alpha,p}(\mathbb{R}^N)\setminus\{0\}$, $t>0$, where
\begin{equation*}
	\sigma = \frac{2+\alpha}{2(p-1)}.
\end{equation*}
Clearly $(H_1)$-$(H_3)$ are satisfied. From \eqref{ineqnorms} we also have $(H_4),(H_5)$. Condition $(H_7)$ was proved in \cite[Lemma 2.10]{PeSi0}, $(H_6)$, $(H_8)$ and $(H_{10})$ are quite obvious, while condition $(H_9)$ follows from \cite[Theorem 1]{MR3568051}. $(H_{11})$ is a consequence of the Pohozaev's identity 
\begin{equation*}\label{POHOZAEV}
	P(u):=	\frac{N-2}{2}\int_{\mathbb{R}^N} |\nabla u|^2dx+\frac{N+\alpha}{2p}\int_{\mathbb{R}^N} |I_{\alpha/2}\star |u|^p|^2dx-\frac{N\lambda}{q}\int_{\mathbb{R}^N} |u|^qdx=0.
\end{equation*}
that holds whenever $I'(u)-\lambda J'(u)=0$ and can be proved along the same lines as in \cite[Proposition 2.5]{MR2902293} (see also \cite[Proposition 5.5]{MR3568051}). Indeed, note that  $\sigma (I'(u)-\lambda J'(u))-P(u)=0$ and hence
\begin{equation*}\label{pohozaev}
	\frac{s}{2}\int_{\mathbb{R}^N} |\nabla u|^2dx+\frac{s}{2p}\int_{\mathbb{R}^N} |I_{\alpha/2}\star |u|^p|^2dx-\frac{s\lambda}{q}\int_{\mathbb{R}^N} |u|^qdx=0,
\end{equation*}
which is exactly $(H_{11})$. One can easily see that $I(u_t)=t^sI(u)$ and $J(u_t)=t^sJ(u)$, which also imply $(H_{12})$, while $(H_{13})$ is trivial.

Now we move on to the conditions concerning $f$ and $g$. First note that \eqref{f1} implies
\begin{equation*}
	\int_{\mathbb{R}^N} \frac{F(u_t)}{t^s}dx=t^{\sigma (\ell-q)}\int_{\mathbb{R}^N} \frac{F(t^\sigma u)}{t^{\sigma \ell}}dx=t^{\sigma (\ell-q)}\int_{\{x\in \mathbb{R}^N:\ u(x)\neq 0\}} \frac{F(t^\sigma u)}{t^{\sigma \ell}|u|^{ \ell}}|u|^{ \ell}dx,\ t>0.
\end{equation*}
 Since $\ell>\tilde{q}$, by \cite[Theorem 5]{MR3568051} we know that $\int_{\mathbb{R}^N} |u|^\ell dx$ is bounded away from zero on weakly compact subsets of $E_{rad}^{\alpha,p}(\mathbb{R}^N)\setminus\{0\}$. Since $\ell<q$, it follows by \eqref{f1} and \eqref{f2} that condition $(F_1)$ is satisfied. Condition $(F_2)$ follows from \eqref{f2}, while condition $(F_3)$ follows from \eqref{f1}, $\ell\in (\tilde{q},2^*)$, and \cite[Theorem 5]{MR3568051}.
 
 To prove $(G_1)$, take $\eta\in (\sigma q-N,\sigma r-N)$ and note that
 \begin{equation*}
 	\frac{\int_{\mathbb{R}^N} |u_t|^r}{|t|^{\eta}}dx=|t|^{\sigma r-N-\eta}\int_{\mathbb{R}^N} |u|^rdx,\ t\in \mathbb{R}. 
 \end{equation*}
 The desired conclusion follows from this and \cite[Theorem 4]{MR3568051}. $(G_2)$ is clear, while $(G_3)$ follows from
 \begin{equation*}
 		\frac{\int_{\mathbb{R}^N} |u_t|^rdx}{t^{\sigma q-N}}=t^{\sigma(r-q)}\int_{\mathbb{R}^N} |u|^rdx,\ t>0
 \end{equation*}
 and \cite[Theorem 5]{MR3568051}. $(G_4)$ follows from \cite[Theorem 4]{MR3568051}. 
 
 By Lemma \ref{PS}, the functional $\Phi$ satisfies the \PS{} condition and so we can apply Theorem \ref{thmlinking} to conclude the proof.

\end{proof}

\newpage
{\bf Acknowledgement.}
This work was completed while the second author held a post-doctoral position at Florida Institute of Technology, Melbourne, United States of
America, supported  by CNPq/Brazil under Grant 201334/2024-0.

\def\cprime{$''$}

\end{document}